\documentclass[sn-mathphys]{sn-jnl}
\usepackage{rotating}
\usepackage{float}
\usepackage{textcomp}
\usepackage{graphicx}
\usepackage{graphics}
\usepackage{tikz}
\usepackage{rotating}
\usepackage{placeins}
\usepackage{array}
\usepackage{booktabs}
\theoremstyle{plain}
\newtheorem{theorem}{Theorem}
\newtheorem{corollary}{Corollary}

\theoremstyle{definition}
\newtheorem{definition}[theorem]{Definition}

\theoremstyle{Remark}
\newtheorem{remark}[theorem]{Remark}

\begin{document}
	
\title[Infinite matrix associated to sequences]{Infinite matrix associated to sequences}

\author*[1]{\fnm{Mamy Laingo Nomenjanahary} \sur{Rakotoarison}}\email{rmlnomenjanahary@gmail.com}

\author[2]{\fnm{Dimbiniaina} \sur{Ratovilebamboavison}}\email{ratovilebamboavison@gmail.com}
\equalcont{These authors contributed equally to this work.}

\author[2]{\fnm{Fanja} \sur{Rakotondrajao}}\email{frakoton@yahoo.fr}
\equalcont{These authors contributed equally to this work.}

\affil[1]{\orgdiv{D\'epartement de Math\'ematiques et informatiques}, \orgname{Universit\'e d'Antananarivo}, \orgaddress{\postcode{101}, \country{Madagascar}}}

\affil[2]{\orgdiv{ISTRALMA Ambatondrazaka}, \orgname{Universit\'e de Toamasina}, \orgaddress{ \city{Toamasina}, \postcode{503}, \country{Madagascar}}}

\affil[2]{\orgdiv{D\'epartement de Math\'ematiques et informatiques}, \orgname{Universit\'e d'Antananarivo}, \orgaddress{\postcode{101}, \country{Madagascar}}}

\abstract{In this paper, we study the Euler-Seidel matrices with coefficients and determine the associated Riordan matrix to a given matrix, if it does exist. Computation of the generating function of the final sequence is established by the associated Riordan matrix. Applications are given.}


\maketitle

\section{Recalls and introduction} \label{sec1}

Dumont \cite{bib1} reintroduced the Euler-Seidel matrix consisting of constructing an infinite matrix, starting from a given sequence $(a_n)$ by setting $a_n^0=a_n$ and defining the sequence $\Bigl(a_n^k\Bigr)_{n,k\geq 0}$ by
\begin{displaymath} 
	a_n^k=a_n^{k-1}+a_{n+1}^{k-1},
\end{displaymath}
represented in Table \ref{tab1}.

\begin{table}[h]
	\centering
	\begin{tabular}{@{}l|cccccl@{}}
		\toprule
		\multicolumn{7}{@{}c@{}}{$a_n^k$}\\
		\midrule
		$k \backslash n$ & 0 & 1 & $\cdots$ & $n$ & $n+1$ & $\cdots$ \\
		\midrule
		0& $a_0^0$ & $a_1^0$ & $\cdots$& $a_n^0$ & $a_{n+1}^0$ & $\cdots$\\
		\ \\
		1& $a_0^1$ & $a_1^1$ & $\cdots$& $a_n^1$ & $a_{n+1}^1$ & $\cdots$\\
		$\vdots$ &$\vdots$ & $\vdots$ & $\ddots$ & $\vdots$ & $\vdots$ & $\ddots$ \\
		$k$-1& $a_0^{k-1}$ & $a_1^{k-1}$ & $\cdots$& $a_n^{k-1}$ & $a_{n+1}^{k-1}$ & $\cdots$\\
		\ \\
		$k$& $a_0^k$ & $a_1^k $& $\cdots$& $a_n^k$ & $a_{n+1}^k$ & $\cdots$\\
		$\vdots$ &$\vdots$	& $\vdots$ & $\ddots$ & $\vdots$ & $\vdots$ & $\ddots$ \\
		\bottomrule 
	\end{tabular}
	\caption{\centering Representation Euler-Seidel matrix}
	\label{tab1}
\end{table}
The sequence $\Bigl(a_{n}^{0}\Bigr)$ of the first row of the matrix $\Bigl(a_{n}^{k}\Bigr)$ is called the initial sequence. The sequence $\Bigl(a_{0}^{k}\Bigr)$ of the first column is called the final sequence, $\Bigl(a_{n}^{k}\Bigr)$ represents the $k$-th row and the $n$-th column of the matrix (the first column and the first row being indexed by $0$). \\
We use $\Bigl(a_{n}^{k}\Bigr)$ matrix notation for Euler-Seidel matrix, where $k$ is the row index and $n$ is the column index. But for Riordan matrix we use $\Bigl(r_{k,\ell}\Bigr)$ matrix notation, where $k$ is the row index and $\ell$ is the column index.\\
Given a sequence $\Bigl(g_{n}\Bigr)_{k\geq 0},$ we define by $g(t)=\displaystyle\sum_{n\geq 0}\ g_n^0\ t^n$ the ordinary generating function of $\Bigl(g_{n}\Bigr)$ and $G(t)=\displaystyle\sum_{n\geq 0}\ g_n^0\ \frac{t^n}{n!}$ the exponential generating function of $\Bigl(g_{n}\Bigr).$ So, we convey to denote an ordinary generating function by an lower case, and to denote an exponential generating function by an upper case.\\
We denote by $a(t)=\displaystyle\sum_{n\geq 0}\ a_n^0\ t^n,\ A(t)=\displaystyle\sum_{n\geq 0}\ a_n^0\ \frac{t^n}{n!},\ \Bar{a}(t)=\displaystyle\sum_{n\geq 0}\ a_0^n\ t^n$ and $\Bar{A}(t)=\displaystyle\sum_{n\geq 0}\ a_0^n\ \frac{t^n}{n!}.$\\

\begin{theorem}[Euler \cite{bib2}]\label{theo11}
	We have
	\begin{displaymath} 
		\Bar{a}(t)=\sum_{n\geq 0}\ a_0^n\ t^n=\frac{1}{1-t}\ a\Bigl(\frac{t}{1-t}\Bigr).
	\end{displaymath}
\end{theorem} 
\begin{theorem}[Seidel \cite{bib5}]\label{theo12}
	If $A(t)=\displaystyle\sum_{n\geq 0}\ a_n^0\ \frac{t^n}{n!}$ is the exponential generating function of the initial sequence. Then the exponential generating function $\Bar{A}(t)$ of the final sequence is determined by
	\begin{equation} 
		\Bar{A}(t)=\sum_{n\geq 0}\ a_0^n\ \frac{t^n}{n!}=A(t)\ \exp(t).\label{equa3}
	\end{equation}
\end{theorem}
Firengiz and Dil \cite{bib3} followed Dumont's idea by adding constant coefficients.
\begin{theorem}[Firengiz and Dil \cite{bib3}]\label{theo13}
	If  $a_n^k=x\ a_n^{k-1}+y\ a_{n+1}^{k-1}$ for given constants $x$ and $y$. Then, we have
	\begin{displaymath} 
		\sum_{n\geq 0}\ \sum_{k\geq 0}\ a_n^k\ \frac{u^k}{k!}\ \frac{t^n}{n!}=\exp(xu)\ A\bigl(t+uy\bigr).
	\end{displaymath}
\end{theorem} 
\begin{definition}[Shapiro, Getu, Woan and Woodson \cite{bib7}, Sprugnoli \cite{bib8}]\label{defi11}
	Consider an infinite matrix \mbox{$R=\Bigl(r_{k,\ell}\Bigr)_{k,\ell\geq 0}$} with complex coefficients and two formal power series \mbox{$g(t)=\displaystyle\sum_{\ell\geq 0}\ g_{\ell}\ t^{\ell},$} \mbox{$f(t)=\displaystyle\sum_{\ell\geq 1}\ f_{\ell}\ t^{\ell}$} with \mbox{$g_0=1,\ f_1\neq 0.$} Let \mbox{$R_{\ell}(t)=\displaystyle\sum_{k\geq 0}\ r_{k,{\ell}}\ t^k$} be the formal power series of the $\ell$-th column of $R$ (the first column being indexed by $0$). Then, $R$ is called a Riordan matrix if
	\begin{equation} 
		R_{\ell}(t)=g(t)\Bigl[f(t)\Bigr]^{\ell}.\label{equ6}
	\end{equation}
	$R$ is an infinite lower-triangular matrix which we denote a pair \mbox{$R=\Bigl(g(t),f(t)\Bigr).$} Note that the identity $I=(1,t)$ is a Riordan matrix.
\end{definition}
\begin{remark}
	 We have $g_k=r_{k,0}$ and
	\begin{equation} 
			R_{\ell}(t)=G(t)\frac{\Bigl[F(t)\Bigr]^{\ell}}{\ell!}.\label{equ66}
	\end{equation}
	We denote $R=\Bigl[G(t),F(t)\Bigr].$
\end{remark}
\begin{theorem}[Shapiro \cite{bib6}, Shapiro, Getu, Woan and Woodson \cite{bib7}] \label{theo4}
	We have
	\begin{equation}
		\Bigl(g(t),f(t)\Bigr)\ 
		\begin{bmatrix}
			a_0\\
			a_1\\
			a_2\\
			\vdots
			\end{bmatrix}=
			\begin{bmatrix}
			b_0\\
			b_1\\
			b_2\\
			\vdots
		\end{bmatrix}
		\iff g(t)\ a\bigl(f(t)\bigr)=b(t). \label{equ7}
	\end{equation}
\end{theorem}
\begin{remark}
	 We have
	\begin{equation}
	\Bigl[G(t),F(t)\Bigr]\ 
	\begin{bmatrix}
	a_0\\
	a_1\\
	a_2\\
	\vdots
	\end{bmatrix}=
	\begin{bmatrix}
	b_0\\
	b_1\\
	b_2\\
	\vdots
	\end{bmatrix}
	\iff G(t)\ A\bigl(F(t)\bigr)=B(t). \label{equ08}
	\end{equation}
\end{remark}
\begin{theorem}[Shapiro \cite{bib6}, Shapiro and al.\cite{bib7}]\label{theo16}
	The Riordan group \mbox{$\mathcal{R}=${\small $\Bigl\{\Bigl(g(t),f(t)\Bigr) \mid \Bigl( g(t), f(t)\Bigr)\ \text{is a Riordan matrix set with } f(t)=t+f_2\ t^2+\ldots\Bigr\}.$}} In other words, each element of $\mathcal{R}$ is a lower triangular matrix with $1$ on the main diagonal. The multiplication in $\mathcal{R}$ is defined by
	\begin{displaymath} 
		\Bigl(g(t),f(t)\Bigr)\ \Bigl(\ell(t),k(t)\Bigr)=\Bigl(g(t)\ \ell\bigl(f(t)\bigr),k\bigl(f(t)\bigr)\Bigr).
	\end{displaymath} 
	The Riordan matrix multiplication identity is $I=(1,t),$ that is
	\begin{displaymath} 
		(1,t)\ \Bigl(g(t),f(t)\Bigr)=\Bigl(g(t),f(t)\Bigr)\ (1,t)=\Bigl(g(t),f(t)\Bigr).
	\end{displaymath}  
		The inverse of $(g(t),f(t))$ is defined by
	\begin{equation} 
		\Bigl(\substack{\dfrac{1}{g\bigl(\bar{f}(t)\bigr)}},\bar{f}(t)\Bigr),\label{equa14}
	\end{equation}
	where $\bar{f}(t)$ is the inverse of $f,$ that is
	\begin{displaymath} 
		f\bigl(\bar{f}(t)\bigr)=\bar{f}\bigl(f(t)\bigr)=t.
	\end{displaymath}
\end{theorem} 

In this paper, we introduce the coefficients $u(n,k)$ and $v(n,k),$ such that $a_n^k=u(n,k)\ a_n^{k-1}+v(n,k)\ a_{n+1}^{k-1}.$
We give formula defining the sequence $\Bigl(a_n^k\Bigr)$ in terms of the initial and final sequences. We establish relations between generating functions of the initial and final sequences. Applications are given to sequences $\Bigl(\delta_n^k\Bigr)$ introduced by \mbox{Rakotondrajao \cite{bib4}} (in table of $d_n^k$). Recall that a permutation is a \mbox{n-fixed-points-permutations} if the set of its fixed points is a subset of $[n]$ and the first $n$ integers are in different cycles. Let denote by $\mathcal{D}_n^k$ the set of $n$-fixed-points-permutations of length $(n+k),\ \lvert \mathcal{D}_n^k\rvert=\delta_n^k.$

\section{Euler-Seidel matrix with two parameter coefficients} \label{sec2}

Let us consider two sequences $u$ and $v$ with two variables $n$ and $k$. We denote by $M(u,v,a)=\Bigl(a_{n}^{k}\Bigr)$ the infinite matrix defined by
\begin{equation} 
	\begin{cases}
	a_n^0&=a_n\ {\bigl(n\geq 0 \bigr)};\\
	a_n^k&=u(n,k)\ a_n^{k-1}+v(n,k)\ a_{n+1}^{k-1}\ {\bigl(n\geq 0,\ k\geq 1\bigr)}.
	\end{cases}\label{equa18}
\end{equation} 	
We give in the following theorem the expressions of the coefficient $a_{n}^{k}$ for a fixed pair $(n,k)$, in terms of given the initial and final sequences.
\begin{theorem}\label{theo21}
	For any integers $n$ and $k$, the coefficient $a_n^k$ is determined by the formula
	\begin{equation} 
		a_n^k=\sum_{\ell=0}^{k}\ C_n(k,\ell)\ a_{n+\ell}^0,\label{equa19}
	\end{equation}
	where
	\begin{equation}
		C_n(k,\ell)=\sum_{\substack{A\subset[k]\\ \lvert A\rvert=k-\ell}}\ \prod_{\substack{0\le i\le k-\ell -1\\ h_i\in A\\ h_i>h_{i+1}}} u(n+k-h_i-i,h_i)\\ \prod_{\substack{0\le j \le \ell -1\\ m_j\in [k]\setminus A\\ m_j>m_{j+1}}} v(n+j,m_j).\label{equa190}
	\end{equation}
\end{theorem}
\begin{proof}
	Recall that each path from $a_{n}^k$ to $a_0^{n+k}$ is identified with a sequence $(s_1,s_2,\ldots,s_k),$ where $s_i=u$ (north direction) or $s_i=v$ (north-east direction), for $1\le i\le k,$ with exactly $(k-\ell)\ u$ and $\ell \ v.$ A such path is therefore identified with a subset $A\subset[k]$ such that $\lvert A\rvert=k-\ell$ and $i\in A\iff s_i=u.$ The north direction from $a_i^j$ to $a_i^{j-1}$ in the path of this network is weighted
	by $u(i,j)$ and the north-east direction from $a_i^j$ to $a_{i-1}^{j+1}$ is weighted by $v(i,j).$ The weight of a such path is
	\begin{displaymath}
		\prod_{\substack{0 \le i \le k- \ell -1\\ h_i\in A\\ h_i > h_{i+1}}} u(n+k-h_i-i,h_i)\ \prod_{\substack{0\le j \le \ell -1\\ m_j\in [k]\setminus A\\ m_j > m_{j+1}}} v(n+j,m_j).
	\end{displaymath}
	Hence the total weight $C_n(k,\ell)$ in the path from $a_n^k$ to vertex $a_{n+\ell}^0$ \mbox{($0\le \ell \le k$)} is equal to
	\begin{displaymath}
		C_n(k,\ell)=\sum_{\substack{A\subset[k]\\ \lvert A\rvert=k-\ell}}\ \prod_{\substack{0 \le i \le k- \ell -1\\ h_i\in A\\ h_i > h_{i+1}}} u(n+k-h_i-i,h_i)\ \prod_{\substack{0\le j \le \ell -1\\ m_j\in [k]\setminus A\\ m_j > m_{j+1}}} v(n+j,m_j).
	\end{displaymath}
	So
	\begin{displaymath}
		a_n^k=\sum_{\ell=0}^{k}\ C_n(k,\ell)\ a_{n+\ell}^0.
	\end{displaymath}
	The Fig. \ref{fig1} represents the rooted network $a_{i}^j$ weighted by $u(i,j)$ and $v(i,j)$, $(i=3,j=4).$
	\begin{figure}[H]
		\centering
		\includegraphics[width=4.5in]{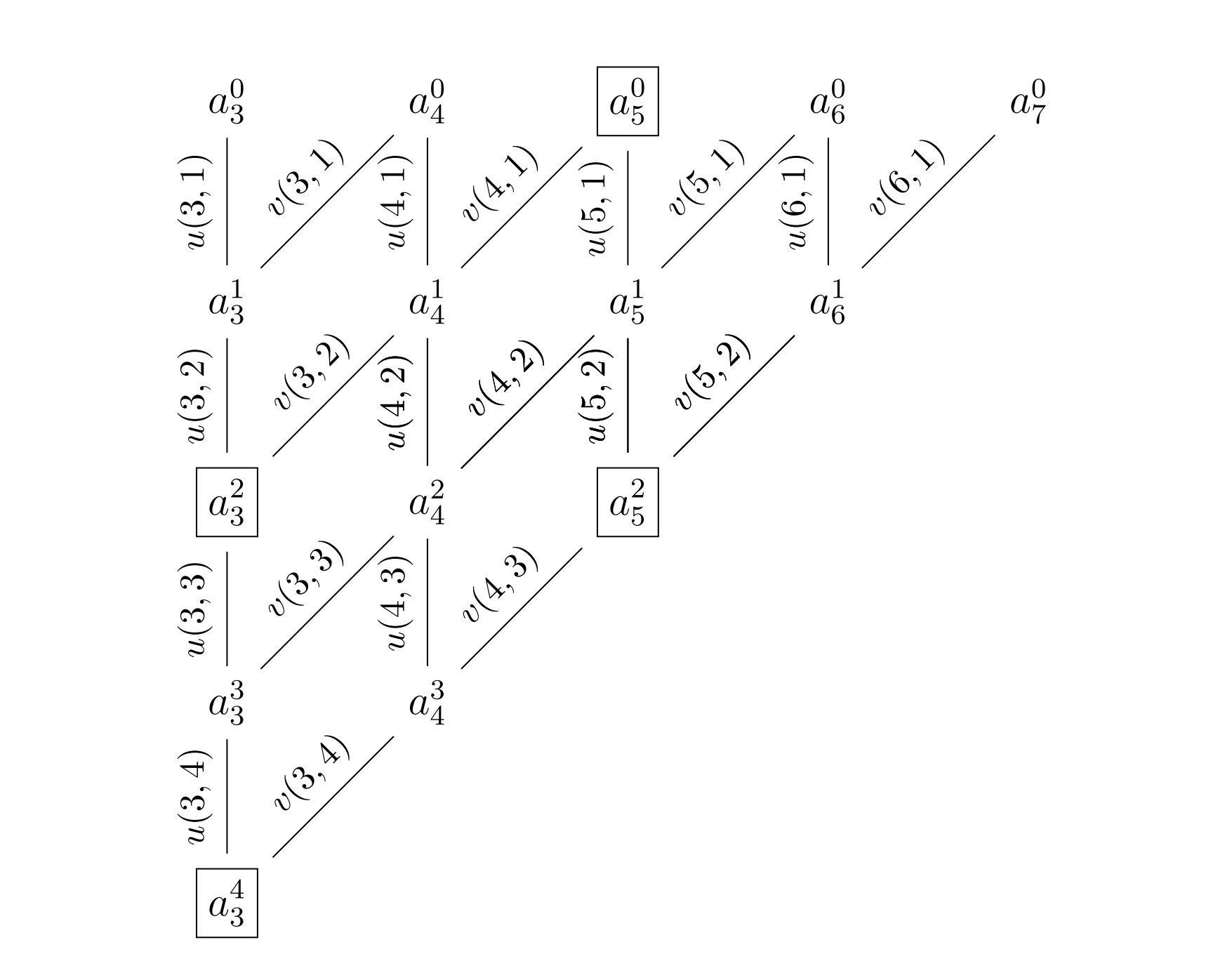}	
		\caption{\centering Rooted network $a_i^j$ weighted by $u(i,j)$ and $v(i,j)$, ($i=3,\ j=4$).}
		\label{fig1}
	\end{figure} 
\end{proof}
\begin{remark}
	We call $C_n=\Bigl(C_n(k,\ell)\Bigr)_{k,\ell\in \mathbb{N}}$ the infinite triangular matrix associated to the column $n$ of $M(u,v,a).$
\end{remark}
\begin{theorem}\label{theo22}
	For any integers $n$ and $k$, we have
	\begin{equation}
		\begin{aligned}
			a_n^k=\sum_{\ell=0}^{n} & \sum_{\substack{A\subset[n]\\ \lvert A\rvert=n-\ell}}\ \prod_{\substack{0 \le i \le n- \ell -1\\ h_i\in A\\ h_i > h_{i+1}}} \frac{-u(k+n-h_i-i+1,h_i-1)}{v(k+n-h_i-i+1,h_i-1)}\\
			&\qquad \ \prod_{\substack{0\le j \le \ell -1\\ m_j\in [n]\setminus A\\ m_j > m_{j+1}}} \frac{1}{v(k+j+1,m_j-1)}\ a_0^{k+\ell}.
		\end{aligned} \label{equa20}
	\end{equation}
\end{theorem}
\begin{proof}
	From equation \eqref{equa18}, we get
	\begin{displaymath}
		a_n^k=\frac{-u(n-1,k+1)}{v(n-1,k+1)}\ a_{n-1}^{k}+\frac{1}{v(n-1,k+1)}\ a_{n-1}^{k+1}.
	\end{displaymath} 
	Let $a_n^k=b_k^n,$ we have
	\begin{displaymath}
		b_k^n=\frac{-u(k+1,n-1)}{v(k+1,n-1)}\ b_{k}^{n-1}+\frac{1}{v(k+1,n-1)}\ b_{k+1}^{n-1}.
	\end{displaymath} 
	In this rooted network $b_k^n$ and vertices $b_{k+\ell}^0,$ for $0\le \ell \le n,$ we assign a weight $\frac{-u(j+1,i-1)}{v(j+1,i-1)}$ on the step from $b_j^i$ to $b_j^{i-1}$ and a weight $\frac{1}{v(j+1,i-1)}$ on the step from $b_j^i$ to $b_{j+1}^{i-1}.$ By similarity to the proof of Theorem \ref{theo21}, the weight of each path from $b_k^n$ to vertex $b_{k+\ell}^0$ is
	\begin{displaymath}
		\prod_{\substack{0\le i\le n- \ell -1\\ h_i\in A\\ h_i > h_{i+1}}} \frac{-u(k+n-h_i-i+1,h_i-1)}{v(k+n-h_i-i+1,h_i-1)} \prod_{\substack{0\le j \le \ell -1\\ m_j\in [n]\setminus A\\ m_j > m_{j+1}}} \frac{1}{v(k+j+1,m_j-1)};
	\end{displaymath}
	where $A\subset[n]$ such that $\lvert A\rvert=n-\ell.$ The total weight of the path from $b_k^n$ to vertex $b_{k+\ell}^0$ is equal to
	\begin{displaymath}
		\sum_{\substack{A\subset[n]\\ \lvert A\rvert =n-\ell}}\ \prod_{\substack{0 \le i\le n- \ell -1\\ h_i\in A\\ h_i > h_{i+1}}} \frac{-u(k+n-h_i-i+1,h_i-1)}{v(k+n-h_i-i+1,h_i-1)}\ \prod_{\substack{0\le j\le \ell -1\\ m_j\in [n]\setminus A\\ m_j > m_{j+1}}} \frac{1}{v(k+j+1,m_j-1)}.
	\end{displaymath}
	Hence,
	\begin{displaymath}
		\begin{aligned}
			b_k^n&=\sum_{\ell=0}^{n}\ \sum_{\substack{A\subset[n]\\ \lvert A\rvert=n-\ell}}\ \prod_{\substack{0 \le i \le n- \ell -1\\ h_i\in A\\ h_i > h_{i+1}}} \frac{-u(k+n-h_i-i+1,h_i-1)}{v(k+n-h_i-i+1,h_i-1)}\\
			&\qquad \ \prod_{\substack{0\le j \le \ell -1\\ m_j\in [n]\setminus A\\ m_j > m_{j+1}}} \frac{1}{v(k+j+1,m_j-1)}\ b_{k+\ell}^0.
		\end{aligned}
	\end{displaymath}
	The Fig. \ref{fig2} represents the rooted network $b_{j}^i$ weighted by $\frac{-u(j+1,i-1)}{v(j+1,i-1)}$ and $\frac{1}{v(j+1,i-1)}$, ($j=3,\ i=4$).
	\begin{figure}[H]
		\centering
		\includegraphics[width=4.5in]{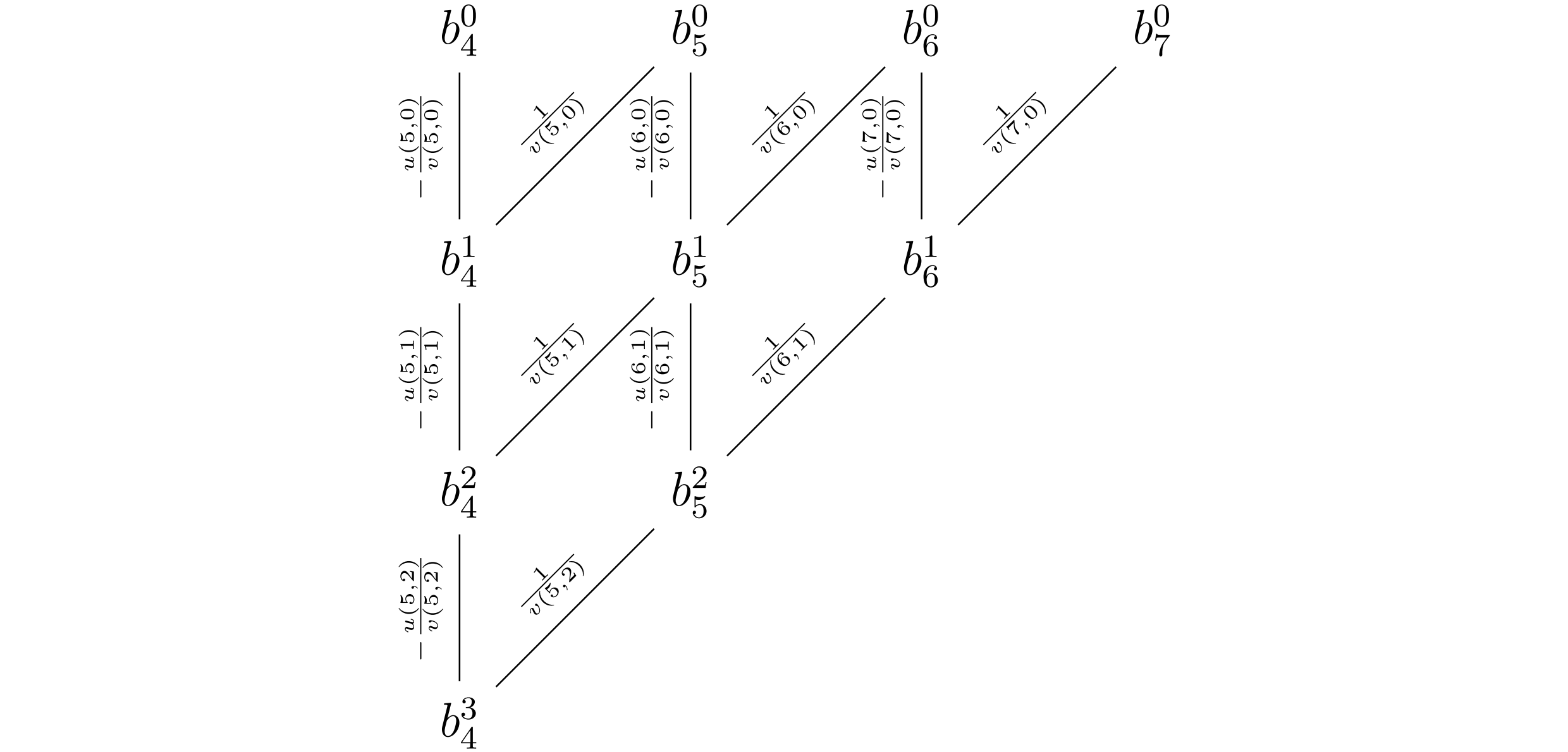}	
		\caption{\centering Root network $b_j^i$ weighted by $\frac{-u(j+1,i-1)}{v(j+1,i-1)}$ and $\frac{1}{v(j+1,i-1)}$, ($j=3,\ i=4$).}
		\label{fig2}
	\end{figure} 
\end{proof}
We have following results.
\begin{theorem}\label{theo23}
	If $a_n^k=u_n\ a_n^{k-1}+v_n\ a_{n+1}^{k-1}$ then 
	\begin{displaymath}
	C_n(k+1,\ell +1)=v_{n+\ell}\ C_n(k,\ell) +\ u_{n+\ell +1}\ C_n(k,\ell+1).
	\end{displaymath}
\end{theorem}
\begin{proof}
	Since there is just one variable $n,$ it suffices to remove the second variable $k$ in equation \eqref{equa190}. We get
		\begin{displaymath}
		\begin{aligned}
			C_n(k+1,\ell +1)&=\sum_{\substack{A\subset[k+1]\\ \lvert A\rvert=k-\ell}}\ \prod_{\substack{0 \le i \le k-\ell -1\\ h_i\in A\\ h_i > h_{i+1}}} u_{n+k+1-h_i-i}\ \prod_{\substack{0 \le j \le \ell}} v_{n+j},\\
			&=\sum_{\substack{A\subset[k]\\ \lvert A\rvert=k-\ell}}\ \prod_{\substack{0 \le i \le k-\ell -1\\ h_i\in A\\ h_i > h_{i+1}}} u_{n+k-h_i-i}\ \prod_{\substack{0 \le j \le \ell}} v_{n+j}\ \ +\\
			&\ \qquad u_{n+\ell +1}\ \sum_{\substack{A\subset[k]\\ \lvert A\rvert=k-\ell-1}}\ \prod_{\substack{0 \le i \le k-\ell -2\\ h_i\in\  A\\ h_i > h_{i+1}}} u_{n+k-h_i-i}\ \prod_{\substack{0 \le j \le \ell}} v_{n+j},\\
			&=v_{n+\ell}\ C_n(k,\ell)+u_{n+\ell +1}\ C_n(k,\ell +1).
		\end{aligned}
	\end{displaymath}
\end{proof}
\begin{corollary}
		If $a_n^k=u_k\ a_n^{k-1}+v_k\ a_{n+1}^{k-1}$ then 
		\begin{displaymath}
			C_n(k+1,\ell +1)=v_{k+1}\ C_n(k,\ell) +\ u_{k+1}\ C_n(k,\ell+1).
		\end{displaymath}
\end{corollary}
\begin{theorem}\label{theo023}
	If $a_n^k=u_k\ a_n^{k-1}+v_n\ a_{n+1}^{k-1}$ then 
	\begin{displaymath}
		C_n(k+1,\ell +1)=v_{n+\ell}\ C_n(k,\ell) +\ u_{k +1}\ C_n(k,\ell+1).
	\end{displaymath}
\end{theorem}
\begin{proof}
	Remove the first variable $n$ in $u$ and the second variable $k$ in $v$. We get
	\begin{displaymath}
		\begin{aligned}
			C_n(k+1,\ell +1)&=\sum_{\substack{A\subset[k+1]\\ \lvert A\rvert=k-\ell}}\ \prod_{\substack{i\in A}} u_{i}\prod_{\substack{0 \le j \le \ell}} v_{n+j},\\
			&=\sum_{\substack{A\subset[k]\\ \lvert A\rvert=k-\ell}}\ \prod_{\substack{i\in A}} u_{i}\ \prod_{\substack{0 \le j \le \ell}} v_{n+j}\ \ +\\
			&\ \qquad u_{k+1}\ \sum_{\substack{A\subset[k]\\ \lvert A\rvert=k-\ell-1}}\ \prod_{\substack{i\in A}} u_{i}\ \prod_{\substack{0 \le j \le \ell}} v_{n+j},\\
			&=v_{n+\ell}\ C_n(k,\ell)+u_{k+1}\ C_n(k,\ell +1).
		\end{aligned}
	\end{displaymath}
\end{proof}
\begin{theorem}\label{theo0023}
	If $a_n^k=u_k\ a_n^{k-1}+v_k\ a_{n+1}^{k-1}$ then 
	\begin{equation}
		C_n(k+1,\ell +1)=v_{k+1}\ C_n(k,\ell) +\ u_{k+1}\ C_n(k,\ell+1).\label{equa0023}
	\end{equation}
\end{theorem}
\begin{proof}
	(By similarity to the proof of Theorem \ref{theo23})\\
	Remove the first variable $n$ in $u$ and $v$. We get
	\begin{displaymath}
		\begin{aligned}
			C_n(k+1,\ell +1)&=\sum_{\substack{A\subset[k+1]\\ \lvert A\rvert=k-\ell}}\ \prod_{\substack{i\in A}} u_{i}\ \prod_{\substack{j\in [k+1]\setminus A}} v_{j},\\
			&=\sum_{\substack{A\subset[k]\\ \lvert A\rvert=k-\ell}}\ \prod_{\substack{i\in A}} u_{i}\prod_{\substack{j\in [k+1]\setminus A}} v_{j}\ \ +\\
			&\ \qquad u_{k+1}\ \sum_{\substack{A\subset[k]\\ \lvert A\rvert=k-\ell-1}}\ \prod_{\substack{i\in A}} u_{i}\ \prod_{\substack{j\in [k]\setminus A}} v_{j},\\
			&=v_{k+1}\ C_n(k,\ell)+u_{k+1}\ C_n(k,\ell +1).
		\end{aligned}
	\end{displaymath}
\end{proof}
\begin{theorem}\label{theo00023}
	If $a_n^k=u(n,k)\ a_n^{k-1}+v_k\ a_{n+1}^{k-1}$ then 
	\begin{equation}
		C_n(k+1,\ell +1)=v_{k+1}\ C_n(k,\ell) +\ u(n,k+1)\ C_n(k,\ell+1).\label{equa00023}
	\end{equation}
\end{theorem}
\begin{proof}
	(By similarity to the proof of Theorem \ref{theo23})\\
	Remove the first variable $n$ in $v$. We get
	\begin{displaymath}
		\begin{aligned}
			C_n(k+1,\ell +1)&=\sum_{\substack{A\subset[k+1]\\ \lvert A\rvert=k-\ell}}\ \prod_{\substack{0 \le i \le k-\ell -1\\ h_i\in A\\ h_i > h_{i+1}}} u_{n+k+1-h_i-i,h_i}\prod_{\substack{j\in [k+1]\setminus A}} v_{j},\\
			&=\sum_{\substack{A\subset[k]\\ \lvert A\rvert=k-\ell}}\ \prod_{\substack{0 \le i \le k-\ell -1\\ h_i\in A\\ h_i > h_{i+1}}} u_{n+k-h_i-i,h_i}\ \prod_{\substack{j\in [k+1]\setminus A}} v_{j}\ \ +\\
			&\ \qquad u_{n,k+1}\ \sum_{\substack{A\subset[k]\\ \lvert A\rvert=k-\ell-1}}\ \prod_{\substack{0 \le i \le k-\ell -2\\ h_i\in\  A\\ h_i > h_{i+1}}} u_{n+k-h_i-i,h_i}\ \prod_{\substack{j\in [k]\setminus A}} v_{j},\\
			&=v_{k+1}\ C_n(k,\ell)+u_{n,k+1}\ C_n(k,\ell +1).
		\end{aligned}
	\end{displaymath}
\end{proof}
\begin{corollary}\label{cor1}
	There exists a unique matrix $R$ such that $\Bigl(a_0^k\Bigr)_{k\geq 0}=R\ \Bigl(a_{k}^0\Bigr)_{k\geq 0}$ with $R=\Bigl(r_{k,\ell}\Bigr)_{k,\ell\geq 0}=C_0.$
\end{corollary}
We denote by $a(t)=\displaystyle\sum_{n\geq 0}\ a_n^0\ t^n,\ A(t)=\displaystyle\sum_{n\geq 0}\ a_n^0\ \frac{t^n}{n!},\ \Bar{a}(t)=\displaystyle\sum_{n\geq 0}\ a_0^n\ t^n$ and $\Bar{A}(t)=\displaystyle\sum_{n\geq 0}\ a_0^n\ \frac{t^n}{n!}.$\\
$R$ is a Riordan matrix if the generating functions of the initial and final sequence are well-defined. These generating functions are related by the following formulas. We say that the Riodan matrix associated to the column 0 of $M(u,v,a).$
\begin{remark}
	If the coefficients $u$ and $v$ do not depend on $n,\ R=C_0=C_n;$ where $R$ is the Riodan matrix associated to each column $n$ of $M(u,v,a).$
\end{remark}
\begin{theorem}\label{theo230}
	The Riordan matrix associated to the column 0 of ${}^t M(u,v,a)$ is the inverse matrix $R^{-1}$ of $R,$ where ${}^t M(u,v,a)$ is the transpose matrix of $M(u,v,a)$ and $R$ is the Riodan matrix associated to the column 0 of $M(u,v,a).$ That is
	\begin{itemize}
		\item if $\displaystyle\sum_{\ell\geq 0}\ r_{k,\ell}\ t^{\ell}=w_{\ell}\ g(t)\ \Bigl(f(t)\Bigr)^{\ell},\ w_{\ell}=1$ or $w_{\ell}=\frac{1}{\ell!},$
		\begin{equation}
			\substack{\dfrac{1}{g\bigl(\bar{f}(t)\bigr)}}\ \Bar{a}\bigl(\bar{f}(t)\bigr)=\begin{cases}
				a(t), &\text{for the case of }\ w_{\ell}=1;\\
				A(t), &\text{for the case of }\ w_{\ell}=\frac{1}{\ell!};
			\end{cases}\label{equ0240}
		\end{equation}
		\item if $\displaystyle\sum_{\ell\geq 0}\ r_{k,\ell}\ \frac{t^{\ell}}{\ell!}=w_{\ell}\ G(t)\ \frac{\Bigl(F(t)\Bigr)^{\ell}}{\ell!},\ w_{\ell}=1$ or $w_{\ell}=\ell!,$
		\begin{equation}
			\substack{\dfrac{1}{G\bigl(\bar{F}(t)\bigr)}}\ \Bar{A}\bigl(\bar{F}(t)\bigr)=\begin{cases}
				a(t), &\text{for the case of }\ w_{\ell}=1;\\
				A(t), &\text{for the case of }\ w_{\ell}=\ell!.
			\end{cases}\label{equ0241}
		\end{equation}
	\end{itemize}
\end{theorem}
\begin{proof}
	Let $\Bigl(b_n^k\Bigr)_{n,k\geq 0}={}^t M(u,v,a)$ be the transpose matrix of $M(u,v,a)$ and $R_1$ be the Riodan matrix associated to the column 0 of ${}^t M(u,v,a).$ We have 
	\begin{displaymath}
		\Bigl(b_0^k\Bigr)_{k\geq 0}=R_1\ \Bigl(b_{k}^0\Bigr)_{k\geq 0}
	\end{displaymath}
	and
	\begin{displaymath}
		\begin{aligned}
			\Bigl(a_0^k\Bigr)_{k\geq 0}&=R\ \Bigl(a_{k}^0\Bigr)_{k\geq 0},\\
			R^{-1}\ \Bigl(a_0^k\Bigr)_{k\geq 0}&=\Bigl(a_{k}^0\Bigr)_{k\geq 0},\\
			R^{-1}\ \Bigl(b_k^0\Bigr)_{k\geq 0}&=\Bigl(b_0^{k}\Bigr)_{k\geq 0}.	
		\end{aligned}
	\end{displaymath}
	From Corollary \ref{cor1}, we get $R_1=R^{-1}.$\\
	For $w_{\ell}=1,\ w_{\ell}=\frac{1}{\ell!}$ or $w_{\ell}=\ell!,$ we have
	\begin{displaymath}
		\begin{aligned}
			\Bigl(a_0^k\Bigr)_{k\geq 0}&=R\ \Bigl(w_k\ a_{k}^0\Bigr)_{k\geq 0},\\
			R^{-1}\ \Bigl(a_0^k\Bigr)_{k\geq 0}&=\Bigl(w_k\ a_{k}^0\Bigr)_{k\geq 0}.
		\end{aligned}
	\end{displaymath}
	\begin{itemize}
		\item If $\displaystyle\sum_{k\geq 0}\ r_{k,\ell}\ t^{k}=w_{\ell}\ g(t)\ \Bigl(f(t)\Bigr)^{\ell},\ w_{\ell}=1$ or $w_{\ell}=\frac{1}{\ell!}$\\
		From equation \eqref{equ7}, we get
		\begin{displaymath}
			\sum_{k\geq 0}\ w_k\ a_{k}^0\ t^k=\substack{\dfrac{1}{g\bigl(\bar{f}(t)\bigr)}}\ \sum_{k\geq 0}\ a_{0}^k\ \Bigl(\bar{f}(t)\Bigr)^{k}.
		\end{displaymath}
		Then we obtain,
		\begin{itemize}
			\item for the case of $w_k=1$,
			\begin{displaymath}
				\begin{aligned}	
					\sum_{k\geq 0}\ a_{k}^0\ t^k&=\substack{\dfrac{1}{g\bigl(\bar{f}(t)\bigr)}}\ \sum_{\ell \geq 0}\ a_{0}^{\ell}\ \Bigl(\bar{f}(t)\Bigr)^{\ell},\\
					a(t)&=\substack{\dfrac{1}{g\bigl(\bar{f}(t)\bigr)}}\ \Bar{a}\bigl(\bar{f}(t)\bigr),
				\end{aligned}
			\end{displaymath}
			where $R^{-1}=\Bigl(\frac{1}{g\bigl(\bar{f}(t)\bigr)},\bar{f}(t)\Bigr);$
			\item for the case of $w_k=\frac{1}{k!}$,
			\begin{displaymath}
				\begin{aligned}	
					\sum_{k\geq 0}\ a_{k}^0\ \frac{t^k}{k!}&=\substack{\dfrac{1}{g\bigl(\bar{f}(t)\bigr)}}\ \sum_{k\geq 0}\ a_{0}^k\ \Bigl(\bar{f}(t)\Bigr)^k,\\
					A(t)&=\substack{\dfrac{1}{g\bigl(\bar{f}(t)\bigr)}}\ \Bar{a}\bigl(\bar{f}(t)\bigr),
				\end{aligned}
			\end{displaymath}
			where $R^{-1}=\Bigl[\frac{1}{g\bigl(\bar{f}(t)\bigr)},\bar{f}(t)\Bigr].$
		\end{itemize}
		\item If $\displaystyle\sum_{k\geq 0}\ r_{k,\ell}\ \frac{t^{k}}{k!}=w_{\ell}\ G(t)\ \frac{\Bigl(F(t)\Bigr)^{\ell}}{\ell!},\ w_{\ell}=1$ or $w_{\ell}=\ell!$\\
		From equation \eqref{equ08}, we get
		\begin{displaymath}
				\sum_{k\geq 0}\ w_k\ a_{k}^0\ \frac{t^k}{k!}=\substack{\dfrac{1}{G\bigl(\bar{F}(t)\bigr)}}\ \sum_{k\geq 0}\ a_{0}^k\ \frac{\Bigl(\bar{F}(t)\Bigr)^{k}}{k!}.
		\end{displaymath}
		Then we obtain,
		\begin{itemize}
			\item for the case of $w_k=1$,
			\begin{displaymath}
				\begin{aligned}	
					\sum_{k\geq 0}\ a_{k}^0\ \frac{t^k}{k!}&=\substack{\dfrac{1}{G\bigl(\bar{F}(t)\bigr)}}\ \sum_{k\geq 0}\ a_{0}^k\ \frac{\Bigl(\bar{F}(t)\Bigr)^{k}}{k!},\\
					A(t)&=\substack{\dfrac{1}{G\bigl(\bar{F}(t)\bigr)}}\ \Bar{A}\bigl(\bar{F}(t)\bigr),
				\end{aligned}
			\end{displaymath}
			where $R^{-1}=\Bigl[\frac{1}{G\bigl(\bar{F}(t)\bigr)},\bar{F}(t)\Bigr];$
			\item for the case of $w_k=k!$,
			\begin{displaymath}
				\begin{aligned}	
					\sum_{k\geq 0}\ a_{k}^0\ t^k&=\substack{\dfrac{1}{G\bigl(\bar{F}(t)\bigr)}}\ \sum_{k\geq 0}\ a_{0}^k\ \frac{\Bigl(\bar{F}(t)\Bigr)^k}{k!},\\
					a(t)&=\substack{\dfrac{1}{G\bigl(\bar{F}(t)\bigr)}}\ \Bar{A}\bigl(\bar{F}(t)\bigr),
				\end{aligned}
			\end{displaymath}
			where $R^{-1}=\Bigl(\frac{1}{G\bigl(\bar{F}(t)\bigr)},\bar{F}(t)\Bigr).$
		\end{itemize}
	\end{itemize}
\end{proof}
\begin{theorem}\label{theo24}
	For given constants $p$ and $q,$ if 
	\begin{displaymath}
		a_n^k=p\ a_n^{k-1}+\frac{q}{n+1}\ a_{n+1}^{k-1}	
	\end{displaymath}
	then 
	\begin{equation}
		\Bar{a}(t)=\frac{1}{1-p\ t}\ A\bigl(\frac{q\ t}{1-p\ t}\bigr).\label{equa24}
	\end{equation}
\end{theorem}
\begin{proof}
	We have
	\begin{displaymath}
		\begin{aligned}
			\sum_{k\geq \ell}\ r_{k,\ell}\ t^k&=\sum_{k\geq \ell}\ \Bigl(\sum_{\substack{A\subset[k]\\ \lvert A\rvert=k-\ell}}\ p^{k-\ell }\ \prod_{\substack{0 \le j \le \ell -1}} \frac{q}{j+1}\Bigr)\ t^k,\\
			&=\frac{q^{\ell}}{\ell!}\ \sum_{k\geq \ell}\ \binom{k}{\ell}\ p^{k-\ell}\ t^k,\\
			&=\frac{\bigl(q\ t\bigr)^{\ell}}{\ell!}\ \Bigl(\sum_{k\geq 0}\ p^{k}\ t^k\Bigr)^{\ell +1},
		\end{aligned}
	\end{displaymath}
	and so we get
	\begin{displaymath}
		\begin{aligned}
			r_{k,\ell}&=\bigl[t^k\bigr]\ \frac{1}{\ell!\ \bigr(1-p\ t\bigr)}\ \Bigl(\frac{q\ t}{1-p\ t}\Bigr)^{\ell},\\
			\sum_{\ell=0}^{k}\ r_{k,\ell}\ a_{\ell}^0&=\sum_{\ell=0}^{k}\ \bigl[t^k\bigr]\ \frac{1}{\ell!\ \bigr(1-p\ t\bigr)}\ \Bigl(\frac{q\ t}{1-p\ t}\Bigr)^{\ell}\ a_{\ell}^0,\\
			\Bigl( a_0^k\Bigr)_{k\geq 0}&=\Bigl(\frac{1}{1-p\ t}, \frac{q\ t}{1-p\ t}\Bigr)\ \Bigl(\frac{1}{k!}\ a_{k}^0\Bigr)_{k \geq 0}.
		\end{aligned}
	\end{displaymath}
	From equation \eqref{equ7}, we obtain
	\begin{displaymath}
		\begin{aligned}
		\Bar{a}(t)&=\frac{1}{1-p\ t}\ \sum_{k\geq 0}\ \frac{1}{k!}\ a_{k}^0\ \bigl(\frac{q\ t}{1-p\ t}\bigr)^{k},\\
		&=\frac{1}{1-p\ t}\ A\bigl(\frac{q\ t}{1-p\ t}\bigr).
	\end{aligned}
\end{displaymath}	
\end{proof}
\begin{theorem}\label{theo024}
	For given constants $p$ and $q,$ if
	\begin{displaymath}
		a_n^k=p\ k\ a_n^{k-1}+ q\ k\ a_{n+1}^{k-1}
	\end{displaymath}
	then 
	\begin{equation}
		\Bar{A}(t)=\frac{1}{1-p\ t}\ a\bigl(\frac{q\ t}{1-p\ t}\bigr).\label{equa024}
	\end{equation}
\end{theorem}
\begin{proof}
	Considering $\Bigl(b_n^k\Bigr)_{n,k\geq 0}={}^t M(u,v,a)$ the transpose matrix of $M(u,v,a)$ such that $a_n^k=p\ k\ a_n^{k-1}+ q\ k\ a_{n+1}^{k-1}$ and $b_n^k=-\frac{p}{q}\ b_n^{k-1}+ \frac{1}{q\ (n+1)}\ b_{n+1}^{k-1}.$ From \mbox{Theorem \ref{theo24}}, we have
	\begin{displaymath}
		\begin{aligned}
			\sum_{k\geq 0}\ b_0^k\ t^k&=\frac{1}{1+\frac{p}{q}\ t}\ \sum_{k\geq 0}\ \frac{1}{k!}\ b_{k}^0\ \Bigl(\frac{\frac{1}{q}\ t}{1+\frac{p}{q}\ t}\Bigr)^{k},\\
			&=\frac{q}{q+p\ t}\ \sum_{k\geq 0}\ \frac{1}{k!}\ b_{k}^0\ \Bigl(\frac{t}{q+p\ t}\Bigr)^{k}.
		\end{aligned}
	\end{displaymath}
	From equation \eqref{equ7}, we get
	\begin{displaymath}
		\begin{aligned}	
			\Bigl(b_0^k\Bigr)_{k\geq 0}&=R\ \Bigl(\frac{1}{k!}\ b_{k}^0\Bigr)_{k\geq 0},\\
			R^{-1}\ \Bigl(b_0^k\Bigr)_{k\geq 0}&=\Bigl(\frac{1}{k!}\ b_{k}^0\Bigr)_{k\geq 0},
		\end{aligned}
	\end{displaymath}
	with $R=\Bigl(g(t),f(t)\Bigr)=\Bigl(\frac{q}{q+p\ t},\frac{t}{q+p\ t}\Bigr)$ and the inverse matrix $R^{-1}$ is defined by
	\begin{displaymath}
		\begin{aligned}
			R^{-1}&=\Bigl[\frac{1}{g\bigl(\bar{f}(t)\bigr)},\bar{f}(t)\Bigr],\\
			&=\Bigl[\frac{1}{\frac{q}{q+p\ \frac{q\ t}{1-p\ t}}},\frac{q\ t}{1-p\ t}\Bigr],\\
			&=\Bigl[\frac{1}{1-p\ t},\frac{q\ t}{1-p\ t}\Bigr].
		\end{aligned}
	\end{displaymath}
	 From Theorem \ref{theo230}, we obtain
	\begin{displaymath}
		\begin{aligned}
			\Bar{A}(t)&=\frac{1}{1-p\ t}\ \sum_{k\geq 0}\ k!\ a_k^0\ \frac{\bigl(\frac{q\ t}{1-p\ t}\bigr)^k}{k!},\\
			&=\frac{1}{1-p\ t}\ a\bigl(\frac{q\ t}{1-p\ t}\bigr).
		\end{aligned}
	\end{displaymath}	
\end{proof}
\begin{theorem}\label{theo2412}
	For given constants $p$ and $q,$ if 
	\begin{displaymath}
		a_n^k=p\ a_n^{k-1}+q\ \bigl(n+1\bigr)\ a_{n+1}^{k-1}
	\end{displaymath}
	then 
	\begin{equation}
		\Bar{A}(t)=\exp\bigl(p\ t\bigr)\ a\bigl(q\ t\bigr).\label{equa2412}
	\end{equation}
\end{theorem}
\begin{proof}
We have
\begin{displaymath}
	\begin{aligned}
		\sum_{k\geq \ell}\ r_{k,\ell}\ \frac{t^k}{k!}&=\sum_{k\geq \ell}\ \Bigl(\sum_{\substack{A\subset[k]\\ \lvert A\rvert=k-\ell}}\ \prod_{\substack{0 \le i \le k- \ell -1\\ h_i\in A\\ h_i > h_{i+1}}}\ p \ \prod_{\substack{0 \le j \le \ell -1}} \ q\ (j+1)\Bigr)\ \frac{t^k}{k!},\\
		&=\sum_{k\geq \ell}\ \Bigl( \binom{k}{k-\ell}\ p^{k-\ell}\ \frac{q^{\ell}\ \ell!}{k!}\Bigr)\ t^k,\\
		&=\bigl(q\ t\bigr)^{\ell}\ \sum_{k\geq \ell}\ p^{k-\ell}\ \frac{t^{k-\ell}}{\bigl(k-\ell\bigr)!},\\
		&=\exp\bigl(p\ t\bigr)\ \bigl(q\ t\bigr)^{\ell},
	\end{aligned}
\end{displaymath}
and so we get
\begin{displaymath}
	\begin{aligned}
		r_{k,\ell}&=\bigl[\frac{t^k}{k!}\bigr]\ \exp\bigl(p\ t\bigr)\ \bigl(q\ t\bigr)^{\ell},\\
		\sum_{\ell=0}^{k}\ r_{k,\ell}\ a_{\ell}^0&=\sum_{\ell=0}^{k}\ \bigl[\frac{t^k}{k!}\bigr]\ \exp\bigl(p\ t\bigr)\ \bigl(q\ t\bigr)^{\ell}\ a_{\ell}^0,\\
		\Bigl( a_0^k\Bigr)_{k\geq 0}&=\Bigl[\exp\bigl(p\ t\bigr),\ q\ t\Bigr]\ \Bigl(k!\ a_{k}^0\Bigr)_{k \geq 0}.
	\end{aligned}
\end{displaymath}
From equation \eqref{equ08}, we obtain
\begin{displaymath}
	\begin{aligned}
		\Bar{A}(t)&=\exp\bigl(p\ t\bigr)\ \sum_{k\geq 0}\ \frac{k!}{k!}\ a_{k}^0\ \Bigl(q\ t\Bigr)^{k},\\
		&=\exp\bigl(p\ t\bigr)\ a\bigl(q\ t\bigr).
	\end{aligned}
\end{displaymath}
\end{proof}	
\begin{theorem}\label{theo024112}
	For given constants $p$ and $q,$ if
	\begin{displaymath}
		a_n^k=\frac{p}{k}\ a_n^{k-1}+ \frac{q}{k}\ a_{n+1}^{k-1}
	\end{displaymath}
	then 
	\begin{equation}
		\Bar{a}(t)=\exp\bigl(p\ t\bigr)\ A\bigl(q\ t\bigr).\label{equa024112}
	\end{equation}
\end{theorem}
\begin{proof}
	Considering $\Bigl(b_n^k\Bigr)_{n,k\geq 0}={}^t M(u,v,a)$ the transpose matrix of $M(u,v,a)$ such that $a_n^k=\frac{p}{k}\ a_n^{k-1}+ \frac{q}{k}\ a_{n+1}^{k-1}$ and $b_n^k=\frac{-p}{q}\ b_n^{k-1}+ \frac{n+1}{q}\ b_{n+1}^{k-1},$ with $q\neq 0.$ From \mbox{Theorem \ref{theo241}}, we have
	\begin{displaymath}
		\begin{aligned}
			\sum_{k\geq 0}\ b_0^k\ t^k&=\exp\bigl(\frac{-p\ t}{q}\bigr)\ a\bigl(\frac{t}{q}\bigr).
		\end{aligned}
	\end{displaymath}
	From equation \eqref{equ08}, we get
	\begin{displaymath}
		\begin{aligned}	
			\Bigl(b_0^k\Bigr)_{k\geq 0}&=R\ \Bigl(k!\ b_{k}^0\Bigr)_{k\geq 0},\\
			R^{-1}\ \Bigl(b_0^k\Bigr)_{k\geq 0}&=\Bigl(k!\ b_{k}^0\Bigr)_{k\geq 0},
		\end{aligned}
	\end{displaymath}
	with $R=\Bigl[G(t),F(t)\Bigr]=\Bigl[\exp\bigl(\frac{-p\ t}{q}\bigr),\ \frac{t}{q}\Bigr]$ and the inverse matrix $R^{-1}$ is defined by
	\begin{displaymath}
		\begin{aligned}
			R^{-1}&=\Bigl(\frac{1}{G\bigl(\bar{F}(t)\bigr)},\bar{F}(t)\Bigr),\\
			&=\Bigl(\exp\bigl(p\ t\bigr),\ q\ t\Bigr).
		\end{aligned}
	\end{displaymath}
	From Theorem \ref{theo230}, we obtain
	\begin{displaymath}
		\begin{aligned}
			\Bar{a}(t)&=\exp\bigl(p\ t\bigr)\ \sum_{k\geq 0}\ \frac{1}{k!}\ a_k^0\ \bigl(q\ t\bigr)^k,\\
			&=\exp\bigl(p\ t\bigr)\ A\bigl(q\ t\bigr).
		\end{aligned}
	\end{displaymath}	
\end{proof}
\begin{theorem}\label{theo241}
	For given constants $p,\ q$ and $s,$ if 
	\begin{displaymath}
		a_n^k=\bigl(p\ n+q\bigr)\ a_n^{k-1}+s\ \bigl(n+1\bigr)\ a_{n+1}^{k-1},\ \text{with }p\neq 0,	
	\end{displaymath}
	then 
	\begin{equation}
		\Bar{A}(t)=\exp\bigl(q\ t\bigr)\ a\bigl(\frac{s}{p}\ \bigl(\exp\bigl(p\ t\bigr)-1\bigr)\bigr).\label{equa241}
	\end{equation}
\end{theorem}
\begin{proof}
	We have
	\begin{displaymath}
		\begin{aligned}
			\sum_{k\geq \ell}\ r_{k,\ell}\ \frac{t^k}{k!}&=\sum_{k\geq \ell}\ \Bigl(\sum_{\substack{A\subset[k]\\ \lvert A\rvert=k-\ell}}\ \prod_{\substack{0 \le i \le k- \ell -1\\ h_i\in A\\ h_i > h_{i+1}}} \bigl(p(k-h_i-i)+q\bigr)\ \prod_{\substack{0 \le j \le \ell -1}} \ s\ (j+1)\Bigr)\ \frac{t^k}{k!},\\
			&=\sum_{k\geq \ell}\ \Bigl(\sum_{\substack{A\subset[k]\\ \lvert A\rvert=k-\ell}}\ \prod_{\substack{0 \le i \le k- \ell -1\\ h_i\in A\\ h_i > h_{i+1}}} \bigl(p(k-h_i-i)+q\bigr)\   \frac{s^{\ell}\ \ell!}{k!}\Bigr)\ t^k,\\
			&=\sum_{k\geq \ell}\ \sum_{j=0}^{k-\ell}\ \frac{q^j}{j!}\ p^{k-\ell -j}\ s^{\ell}\  \sum_{\substack{k_1+k_2+\cdot+k_{\ell}=k-j\\ k_1,k_2,\cdots,k_{\ell}\geq 1}}\ \frac{1}{k_1!\ k_2! \cdots k_{\ell}!}\ t^k,\\
			&=\Bigl(\sum_{k\geq 0}\ q^k\ \frac{t^k}{k!}\Bigr)\ \Bigl(\sum_{k\geq \ell}\ p^{k-\ell}\ s^{\ell}\  \sum_{\substack{k_1+k_2+\cdot+k_{\ell}=k\\ k_1,k_2,\cdots,k_{\ell}\geq 1}}\ \frac{1}{k_1!\ k_2! \cdots k_{\ell}!}\ t^k\Bigr),\\
			&=\exp\bigl(q\ t\bigr)\ \Bigl(\frac{s}{p}\ \bigl(\exp\bigl(p\ t\bigr)-1\bigr)\Bigr)^{\ell},
		\end{aligned}
	\end{displaymath}
	and so we get
	\begin{displaymath}
		\begin{aligned}
			r_{k,\ell}&=\bigl[\frac{t^k}{k!}\bigr]\ \exp\bigl(q\ t\bigr)\ \Bigl(\frac{s}{p}\ \bigl(\exp\bigl(p\ t\bigr)-1\bigr)\Bigr)^{\ell},\\
			\sum_{\ell=0}^{k}\ r_{k,\ell}\ a_{\ell}^0&=\sum_{\ell=0}^{k}\ \bigl[\frac{t^k}{k!}\bigr]\ \exp\bigl(q\ t\bigr)\ \Bigl(\frac{s}{p}\ \bigl(\exp\bigl(p\ t\bigr)-1\bigr)\Bigr)^{\ell}\ a_{\ell}^0,\\
			\Bigl( a_0^k\Bigr)_{k\geq 0}&=\Bigl[\exp\bigl(q\ t\bigr),\ \frac{s}{p}\ \bigl(\exp\bigl(p\ t\bigr)-1\bigr)\Bigr]\ \Bigl(k!\ a_{k}^0\Bigr)_{k \geq 0}.
		\end{aligned}
	\end{displaymath}
	From equation \eqref{equ08}, we obtain
	\begin{displaymath}
		\begin{aligned}
			\Bar{A}(t)&=\exp\bigl(q\ t\bigr)\ \sum_{k\geq 0}\ \frac{k!}{k!}\ a_{k}^0\ \Bigl(\frac{s}{p}\ \bigl(\exp\bigl(p\ t\bigr)-1\bigr)\Bigr)^{k},\\
			&=\exp\bigl(q\ t\bigr)\ a\bigl(\frac{s}{p}\ \bigl(\exp\bigl(p\ t\bigr)-1\bigr)\bigr).
		\end{aligned}
	\end{displaymath}
\end{proof}	
\begin{theorem}\label{theo02411}
	For given constants $p,\ q$ and $s,$ if
	\begin{displaymath}
		a_n^k=\frac{p\ k+q}{k}\ a_n^{k-1}+ \frac{s}{k}\ a_{n+1}^{k-1}
	\end{displaymath}
	then 
	\begin{equation}
		\Bar{a}(t)=\exp\Bigl(\frac{p+q}{p}\ \bigl(\ln (-p\ t)+1\bigr)\Bigr)\ A\bigl(\frac{-s}{p}\ \bigl(\ln (-p\ t)+1\bigr)\bigr).\label{equa02411}
	\end{equation}
\end{theorem}
\begin{proof}
	Considering $\Bigl(b_n^k\Bigr)_{n,k\geq 0}={}^t M(u,v,a)$ the transpose matrix of $M(u,v,a)$ such that $a_n^k=\frac{p\ k+q}{k}\ a_n^{k-1}+ \frac{s}{k}\ a_{n+1}^{k-1}$ and $b_n^k=\frac{-p\ n -(q+p)}{s}\ b_n^{k-1}+ \frac{n+1}{s}\ b_{n+1}^{k-1},$ with $p\neq 0$ and $s\neq 0.$ From \mbox{Theorem \ref{theo241}}, we have
	\begin{displaymath}
		\begin{aligned}
			\sum_{k\geq 0}\ b_0^k\ \frac{t^k}{k!}&=\exp\bigl(\frac{-(q+p)\ t}{s}\bigr)\ \sum_{k\geq 0}\ b_{k}^0\ \Bigl(\frac{-1}{p}\ \bigl(\exp\bigl(\frac{-p\ t}{p}\bigr)-1\bigr)\Bigr)^{k}.
		\end{aligned}
	\end{displaymath}
	From equation \eqref{equ08}, we get
	\begin{displaymath}
		\begin{aligned}	
			\Bigl(b_0^k\Bigr)_{k\geq 0}&=R\ \Bigl(k!\ b_{k}^0\Bigr)_{k\geq 0},\\
			R^{-1}\ \Bigl(b_0^k\Bigr)_{k\geq 0}&=\Bigl(k!\ b_{k}^0\Bigr)_{k\geq 0},
		\end{aligned}
	\end{displaymath}
	with $R=\Bigl[G(t),F(t)\Bigr]=\Bigl[\exp\bigl(\frac{-(q+p)\ t}{s}\bigr),\ \frac{-1}{p}\ \bigl(\exp\bigl(\frac{-p\ t}{p}\bigr)-1\bigr)\Bigr]$ and the inverse matrix $R^{-1}$ is defined by
	\begin{displaymath}
		\begin{aligned}
			R^{-1}&=\Bigl(\frac{1}{G\bigl(\bar{F}(t)\bigr)},\bar{F}(t)\Bigr),\\
			&=\Bigl(\exp\Bigl(\frac{p+q}{p}\ \bigl(\ln (-p\ t)+1\bigr)\Bigr),\frac{-s}{p}\ \bigl(\ln (-p\ t)+1\bigr)\Bigr).
		\end{aligned}
	\end{displaymath}
	From Theorem \ref{theo230}, we obtain
	\begin{displaymath}
		\begin{aligned}
			\Bar{a}(t)&=\exp\Bigl(\frac{p+q}{p}\ \bigl(\ln (-p\ t)+1\bigr)\Bigr)\ \sum_{k\geq 0}\ \frac{1}{k!}\ a_k^0\ \Bigl(\frac{-s}{p}\ \bigl(\ln (-p\ t)+1\bigr)\Bigr)^k,\\
			&=\exp\Bigl(\frac{p+q}{p}\ \bigl(\ln (-p\ t)+1\bigr)\Bigr)\ A\bigl(\frac{-s}{p}\ \bigl(\ln (-p\ t)+1\bigr)\bigr).
		\end{aligned}
	\end{displaymath}	
\end{proof}
\begin{theorem}\label{theo24111}
	For given constants $p,\ q$ and $s,$ if 
	\begin{displaymath}
		a_n^k=\frac{p\ n+q}{k}\ a_n^{k-1}+ \frac{s}{k}\ a_{n+1}^{k-1},\ \text{with }p\neq 0,	
	\end{displaymath}
	then 
	\begin{equation}
		\Bar{a}(t)=\exp\bigl(q\ t\bigr)\ A\bigl(\frac{s}{p}\ \bigl(\exp\bigl(p\ t\bigr)-1\bigr)\bigr).\label{equa24111}
	\end{equation}
\end{theorem}
\begin{proof}
	We have
	\begin{displaymath}
		\begin{aligned}
			\sum_{k\geq \ell}\ r_{k,\ell}\ t^k&=\sum_{k\geq \ell}\ \Bigl(\sum_{\substack{A\subset[k]\\ \lvert A\rvert=k-\ell}}\ \prod_{\substack{0 \le i \le k- \ell -1\\ h_i\in A\\ h_i > h_{i+1}}} \frac{p(k-h_i-i)+q}{h_i}\ \prod_{j\in [k]\setminus A} \frac{s}{j}\ t^k,\\
			&=\sum_{k\geq \ell}\ \Bigl(\sum_{\substack{A\subset[k]\\ \lvert A\rvert=k-\ell}}\ \prod_{\substack{0 \le i \le k- \ell -1\\ h_i\in A\\ h_i > h_{i+1}}} \bigl(p(k-h_i-i)+q\bigr)\   \frac{s^{\ell}}{k!}\Bigr)\ t^k,\\
			&=\sum_{k\geq \ell}\ \sum_{j=0}^{k-\ell}\ \frac{q^j}{j!\ \ell!}\ p^{k-\ell -j}\ s^{\ell}\  \sum_{\substack{k_1+k_2+\cdot+k_{\ell}=k-j\\ k_1,k_2,\cdots,k_{\ell}\geq 1}}\ \frac{1}{k_1!\ k_2! \cdots k_{\ell}!}\ t^k,\\
			&=\frac{1}{\ell!}\ \Bigl(\sum_{k\geq 0}\ q^k\ \frac{t^k}{k!}\Bigr)\ \Bigl(\sum_{k\geq \ell}\ p^{k-\ell}\ s^{\ell}\  \sum_{\substack{k_1+k_2+\cdot+k_{\ell}=k\\ k_1,k_2,\cdots,k_{\ell}\geq 1}}\ \frac{1}{k_1!\ k_2! \cdots k_{\ell}!}\ t^k\Bigr),\\
			&=\frac{1}{\ell!}\ \exp\bigl(q\ t\bigr)\ \Bigl(\frac{s}{p}\ \bigl(\exp\bigl(p\ t\bigr)-1\bigr)\Bigr)^{\ell},
		\end{aligned}
	\end{displaymath}
	and so we get
	\begin{displaymath}
		\begin{aligned}
			r_{k,\ell}&=\bigl[t^k\bigr]\ \frac{1}{\ell!}\ \exp\bigl(q\ t\bigr)\ \Bigl(\frac{s}{p}\ \bigl(\exp\bigl(p\ t\bigr)-1\bigr)\Bigr)^{\ell},\\
			\sum_{\ell=0}^{k}\ r_{k,\ell}\ a_{\ell}^0&=\sum_{\ell=0}^{k}\ \bigl[t^k\bigr]\ \frac{1}{\ell!}\ \exp\bigl(q\ t\bigr)\ \Bigl(\frac{s}{p}\ \bigl(\exp\bigl(p\ t\bigr)-1\bigr)\Bigr)^{\ell}\ a_{\ell}^0,\\
			\Bigl( a_0^k\Bigr)_{k\geq 0}&=\Bigl(\exp\bigl(q\ t\bigr),\ \frac{s}{p}\ \bigl(\exp\bigl(p\ t\bigr)-1\bigr)\Bigr)\ \Bigl(\frac{1}{k!}\ a_{k}^0\Bigr)_{k \geq 0}.
		\end{aligned}
	\end{displaymath}
	From equation \eqref{equ7}, we obtain
	\begin{displaymath}
		\begin{aligned}
			\Bar{a}(t)&=\exp\bigl(q\ t\bigr)\ \sum_{k\geq 0}\ \frac{1}{k!}\ a_{k}^0\ \Bigl(\frac{s}{p}\ \bigl(\exp\bigl(p\ t\bigr)-1\bigr)\Bigr)^{k},\\
			&=\exp\bigl(q\ t\bigr)\ A\bigl(\frac{s}{p}\ \bigl(\exp\bigl(p\ t\bigr)-1\bigr)\bigr).
		\end{aligned}
	\end{displaymath}
\end{proof}	
\begin{theorem}\label{theo0241123}
	For given constants $p,\ q$ and $s,$ if
	\begin{displaymath}
		a_n^k=\bigl(p\ k+q\bigr)\ a_n^{k-1}+ s\ \bigl(n+1\bigr)\ a_{n+1}^{k-1}
	\end{displaymath}
	then 
	\begin{equation}
		\Bar{A}(t)=\exp\Bigl(\frac{p+q}{p}\ \bigl(\ln (-p\ t)+1\bigr)\Bigr)\ a\bigl(\frac{-s}{p}\ \bigl(\ln (-p\ t)+1\bigr)\bigr).\label{equa0241123}
	\end{equation}
\end{theorem}
\begin{proof}
	Considering $\Bigl(b_n^k\Bigr)_{n,k\geq 0}={}^t M(u,v,a)$ the transpose matrix of $M(u,v,a)$ such that $a_n^k=\bigl(p\ k+q\bigr)\ a_n^{k-1}+ s\ \bigl(n+1\bigr)\ a_{n+1}^{k-1}$ and $b_n^k=\frac{-p\ n -(q+p)}{s\ k}\ b_n^{k-1}+ \frac{1}{s\ k}\ b_{n+1}^{k-1},$ with $p\neq 0$ and $s\neq 0.$ From \mbox{Theorem \ref{theo24111}}, we have
	\begin{displaymath}
		\begin{aligned}
			\sum_{k\geq 0}\ b_0^k\ t^k&=\exp\bigl(\frac{-(q+p)}{s}\ t\bigr)\ \sum_{k\geq 0}\ b_{k}^0\ \Bigl(\frac{-1}{p}\ \bigl(\exp\bigl(\frac{-p\ t}{p}\bigr)-1\bigr)\Bigr)^{k}.
		\end{aligned}
	\end{displaymath}
	From equation \eqref{equ08}, we get
	\begin{displaymath}
		\begin{aligned}	
			\Bigl(b_0^k\Bigr)_{k\geq 0}&=R\ \Bigl(k!\ b_{k}^0\Bigr)_{k\geq 0},\\
			R^{-1}\ \Bigl(b_0^k\Bigr)_{k\geq 0}&=\Bigl(k!\ b_{k}^0\Bigr)_{k\geq 0},
		\end{aligned}
	\end{displaymath}
	with $R=\Bigl(g(t),f(t)\Bigr)=\Bigl(\exp\bigl(\frac{-(q+p)}{s}\ t\bigr),\ \frac{-1}{p}\ \bigl(\exp\bigl(\frac{-p\ t}{p}\bigr)-1\bigr)\Bigr)$ and the inverse matrix $R^{-1}$ is defined by
	\begin{displaymath}
		\begin{aligned}
			R^{-1}&=\Bigl[\frac{1}{g\bigl(\bar{f}(t)\bigr)},\bar{f}(t)\Bigr],\\
			&=\Bigl[\exp\Bigl(\frac{-(p+q)}{p}\ \bigl(\ln (-p\ t)+1\bigr)\Bigr),\frac{-s}{p}\ \bigl(\ln (-p\ t)+1\bigr)\Bigr].
		\end{aligned}
	\end{displaymath}
	From Theorem \ref{theo230}, we obtain
	\begin{displaymath}
		\begin{aligned}
			\Bar{A}(t)&=\exp\Bigl(\frac{p+q}{p}\ \bigl(\ln (-p\ t)+1\bigr)\Bigr)\ \sum_{k\geq 0}\ \frac{k!}{k!}\ a_k^0\ \Bigl(\frac{-s}{p}\ \bigl(\ln (-p\ t)+1\bigr)\Bigr)^k,\\
			&=\exp\Bigl(\frac{p+q}{p}\ \bigl(\ln (-p\ t)+1\bigr)\Bigr)\ a\bigl(\frac{-s}{p}\ \bigl(\ln (-p\ t)+1\bigr)\bigr).
		\end{aligned}
	\end{displaymath}	
\end{proof}

\section{Applications}\label{sec3}

\subsection{Case of $U_n=-1$ and $v_n=n+1$} \label{subsec31}

Let us consider the matrix $\Bigl(\delta_n^k\Bigr),$ where $\delta_n^k$ is the general coefficient of the Euler-Seidel matrix with $u_n=-1$ and $v_n=n+1,$ (\text{see table \ref{tab2}})
\begin{equation}
\left\{ \begin{array}{ll}
\delta_n^0=1\ &(n\geq 0)\\
\delta_n^k=-\delta_n^{k-1}+(n+1)\ \delta_{n+1}^{k-1}\ &(k\geq 1,\ n\geq 0).
\end{array}\right.\label{equa42}
\end{equation}

\begin{table}[h!]
	\centering
	\begin{tabular}{@{}l|cccccl@{}}
		\toprule
		\multicolumn{7}{@{}c@{}}{$\delta_n^k$}\\
		\midrule
		$k \backslash n$ & 0 & 1 & 2 & 3 & 4 & $\cdots$ \\
		\midrule
		0 & 1 & 1 & 1 & 1 & 1 & $\cdots$\\
		1 & 0 & 1 & 2 & 3 & $\cdots$ & $\cdots$\\
		2 & 1 & 3 & 7 & $\cdots$ & $\cdots$ & $\cdots$ \\
		3 & 2 & 11& $\cdots$& $\cdots$ & $\cdots$ & $\cdots$\\
		4 & 9 & $\cdots $& $\cdots$& $\cdots$ & $\cdots$ & $\cdots$\\
		$\vdots$ & $\vdots$	& $\vdots$ & $\vdots$ & $\vdots$ & $\vdots$ & $\ddots$ \\
		\botrule
	\end{tabular}
	\caption{\centering Matrix $\Bigl(\delta_n^k\Bigr)$}\label{tab2}
\end{table}	

\begin{theorem}\label{theo31}
	The exact value of the number of n-fixed-points-permutations of length $(n+k)$ is equal to
	\begin{equation}
	\delta_n^k=\sum_{\ell=0}^{k}\ (-1)^{k-\ell}\ \binom{k}{\ell}\ \frac{(n+\ell)!}{n!}. \label{equa43}
	\end{equation}
\end{theorem}

\begin{proof}
	From equation (\ref{equa19}), we have
	\begin{displaymath}
	\begin{aligned}
		\delta_n^k&=\sum_{\ell=0}^{k}\ \sum_{\substack{A\subset[k]\\ \lvert A\rvert=\ell}}\ \prod_{i=1}^{\ell}\ (n+i)\ (-1)^{k-\ell}\ \delta_{n+\ell}^0,\\
		&=\sum_{\ell=0}^{k}\ \binom{k}{\ell}\ \frac{(n+\ell)!}{n!}\ (-1)^{k-\ell}.
	\end{aligned}
	\end{displaymath}
\end{proof}	
Note that if $n=0,$ we obtain
\begin{equation}
	\begin{aligned}
		\delta_0^k&=\sum_{\ell=0}^{k}\ (-1)^{k-\ell}\ \frac{k!}{(k-\ell)!},\\
		&=d_k, \label{equa44}
	\end{aligned}
\end{equation}
where $d_k$ is the classical number of derangement of length $k$.\\
From Theorem \ref{theo2412}, we have	
\begin{eqnarray}
	\begin{aligned}
		\Bar{A}(t)&=\exp(-t)\ \frac{1}{1-t}\label{equa0462}\\
		\text{et}\ R&=\Bigl[\exp(-t),\ t\Bigr] .\label{equa464}
	\end{aligned}
\end{eqnarray}

\subsection{Case of $u(n,k)=\frac{-(n+1)}{k}$ and $V(n,k)=V_k=\frac{1}{k}$}\label{subsec32}

Let us consider the matrix $\Bigl(a_n^k\Bigr),$ (\text{see table \ref{tab3}})
\begin{equation}
	\left\{ \begin{array}{ll}
		a_n^0=1\ &(n\geq 0)\\
		a_n^k=-\frac{n+1}{k}\ a_n^{k-1}+\frac{1}{k}\ a_{n+1}^{k-1}\ &(k\geq 1,\ n\geq 0).
	\end{array}\right.\label{equa4211}
\end{equation}

\begin{table}[h!]
	\centering
	\begin{tabular}{@{}l|cccccl@{}}
		\toprule
		\multicolumn{7}{@{}c@{}}{$\delta_n^k$}\\
		\midrule
		$k \backslash n$ & 0 & 1 & 2 & 3 & 4 & $\cdots$ \\
		\midrule
		0 & 1 & 2 & 7 & 35 & 228 & $\cdots$\\
		1 & 1 & 3 & 14 & 88 & $\cdots$ & $\cdots$\\
		2 & 1 & 4 & 23 & $\cdots$ & $\cdots$ & $\cdots$ \\
		3 & 1 & 5& $\cdots$& $\cdots$ & $\cdots$ & $\cdots$\\
		4 & 1 & $\cdots $& $\cdots$& $\cdots$ & $\cdots$ & $\cdots$\\
		$\vdots$ & $\vdots$	& $\vdots$ & $\vdots$ & $\vdots$ & $\vdots$ & $\ddots$ \\
		\botrule
	\end{tabular}
	\caption{\centering Matrix $\Bigl(\delta_n^k\Bigr)$}\label{tab3}
\end{table}	

From Theorem \ref{theo24111}, we have	
\begin{eqnarray}
	\begin{aligned}
		\Bar{a}(t)&=\frac{1}{1+\ln \Bigl(\frac{1}{1+\exp(t)}\Bigr)}\label{equa046}\\
		\text{et}\ R&=\Bigl[\exp(-t),\ 1-\exp(-t)\Bigr] .\label{equa46}
	\end{aligned}
\end{eqnarray}
	
\section{Conclusion} \label{sec4}
In this Euler-Seidel matrix, we used a weighted network for the calculation of each element of the matrix by assigning a weight $u(n,k)$ which is the coefficient of $\Bigl(a_n^{k-1}\Bigr)$ on each North direction and a weight $v(n,k)$ which is the coefficient of $\Bigl(a_{n+1}^{k-1}\Bigr)$ on each North-East direction. The generating function of the final sequence is determined from the use of the associated Riordan matrix.

\section*{Statements and Declarations}
\noindent{\bf Funding}: No funding was received for conducting this study or for the preparation of this manuscript.\\

\noindent{\bf Conflict of Interest}: The  authors declare that they have no competing interests.\\

\noindent{\bf Ethical Approval}: Not applicable.\\

\noindent{\bf Data availability}: Data sharing not applicable as no datasets were generated.

\end{document}